\documentclass[12 pt]{paper}
\usepackage{mathrsfs}
\usepackage{amsmath}
\usepackage{graphicx}
\usepackage{xcolor}
\usepackage{bm}
\usepackage{amssymb}
\usepackage{relsize}
\usepackage{tikz-cd}
= 600pt \footskip=40pt\evensidemargin=10pt

\usepackage{amsthm}

\newcounter{eq}
\toks0 \expandafter{\section}
\edef\section{\noexpand\setcounter{eq}{0}\the\toks0}

\numberwithin{equation}{section}

\theoremstyle{plain}
\newtheorem{theorem}{Theorem}
\newtheorem{lemma}{Lemma}
\newtheorem{corollary}{Corollary}
\newtheorem{proposition}{Proposition}
\theoremstyle{definition}
\newtheorem{definition}{Definition}
\newtheorem{remark}{Remark}
\newtheorem{notation}{Notation}
\title{Cohomology and deformations of Jacobi-Jordan algebras}
\author{ Amir Baklouti $^1$, Sa\"{i}d Benayadi $^2$,\\ Abdenacer Makhlouf $^3$, Sabeur Mansour $^4$, \\ \footnotesize $^{1,4}$Umm Al-Qura University, College of first Common year, Department of mathematics, P.O. Box 14035, \\ \footnotesize  Holly Makkah 21955, Saudi Arabia\\\footnotesize 
$^2$ Universit\'e de Lorraine, Laboratoire IECL, CNRS UMR 7502, UFR MIM, 3 rue Augustin Frenel BP 45112,\\ \footnotesize F-57112  Metz
Cedex 03, France\\\footnotesize
$^3$ Universit\'e de Haute Alsace, IRIMAS - D\'epartement de Math\'ematiques, 6 rue des fr\`eres Lumi\`ere \\ \footnotesize F-68093 Mulhouse, France}

\begin{document}
\newtheorem{Example}{Example}
\title{Cohomology and deformations of Jacobi-Jordan algebras}

\maketitle

%
%

\abstract
    In this paper, we define and develop a cohomology and deformation theories of Jacobi-Jordan algebras. We construct a cohomology based on two operators, called zigzag cohomology,  and detail the low degree cohomology spaces. We describe the relationships between  first and second cohomology groups  with extensions and deformations. Moreover, we consider cohomology properties of pseudo-euclidean Jacobi-Jordan algebras and  provide a deformation theory that fits with our zigzag cohomology of Jacobi-Jordan algebras. Furthermore, the paper  includes several examples and applications.
    
\\
\begin{small}MSC[2020] {17C50,16W10,17B56}\end{small}

\section{Introduction}
Cohomology theories  of associative and nonassociative algebras have been extensively studied.  In the context of Jordan algebras the results are relatively scarce. Most of the known results focus on the theory of Lie algebras, see  \cite{fialowski,fucks,harisson,neeb,vagmann}. In this paper we consider 
a class of commutative non-associative algebras that satisfy
the Jacobi identity instead of associativity, which are called Jacobi-Jordan algebras. It turns out that they are a special class of Jordan
algebras. The set of these algebras
is a subclass of a class of Jordan-Lie superalgebras introduced in \cite{kamiya}. For more details about these algebras see \cite{amir2} and \cite{burde}.

Our main upshot in this work is developing a cohomology and deformation theory for Jacobi-Jordan algebras
analogous to the existing theories for associative and Lie algebras. Notice that there is no general cohomology for Jordan algebras, even though there is some attempts, see \cite{Chu,MDV,Glassman,Glassman2}. 
However,
we introduce a  cohomology that has several properties not enjoyed by the Hochschild theory. 
The complex provided here is defined by  two sequences of operators  $d^i$ and $\delta^i$ which satisfy $d^p\delta^{p-1}=0$ for any integer $p$. We denote by $ ~_zZ^p
({\mathfrak  J}, M) $ the kernel of $d$ and by $ ~_zB^p
({\mathfrak  J}, M) $ the image of $\delta$ and by $~_zH^p(\mathfrak  J,M) := ~_zZ^p(\mathfrak  J,M) /~_zB^p(\mathfrak  J,M)$ the $p^{\rm th}$ cohomology space of $\mathfrak  J$ with values/coefficients in Jacobi-Jordan $\mathfrak J$-module $ M$.
We obtain that 
$~_zH^0(\mathfrak  J,M)$ is the annihilator of $\mathfrak  J$ and $~_zH^1(\mathfrak  J,M)$ corresponds to derivations. Moreover, the second group of cohomology classifies the extension of $\mathfrak  J$ by a module $M$. Furthermore, we show that this cohomology controls one parameter formal deformations. We also explore  cohomology properties of pseudo-euclidean Jacobi-Jordan algebras. 

The outline of the paper is as follows: In the second section we give some basic notions and concepts used through the paper, namely some properties of Jacobi-Jordan algebras and their representations. In Section \ref{Sec3}, we introduce a cohomology for Jacobi-Jordan algebras. It is called zigzag cohomology since it deals with two types of cochains and two sequences of operators. We explore the low degree cohomology spaces and extensions.  Moreover, explicit computation on examples are provided. 
in Section \ref{Sec4}, we study scalar zigzag cohomology of pseudo-Euclidean Jacobi-Jordan algebras, which are Jacobi-Jordan algebras endowed with an associative nondegenerate,  symmetric bilinear forms. We characterize the second cohomology group in terms of derivations and antiderivations. Moreover, we provide an exact sequence involving the third cohomology group, from which we deduce some dimension results. 
In Section \ref{Sec5}, we study one-parameter formal deformations of Jacobi-Jordan algebras and show that the zigzag cohomology fits perfectly and provide the expected results, that is the infinitesimal deformations are in one-to-one correspondence with the elements of the second zigzag cohomology group. Section \ref{Sec6} deals with deformations of Jacobi-Jordan algebra homomorphisms with similar results.  We prove that infinitesimal deformations of a Jacobi-Jordan  algebra
homomorphism from $\mathfrak  J$  into an admissible Jacobi-Jordan algebra $M$ are classified by the first cohomology group $~_zH^1({\mathfrak  J},A)$. The last section is devoted to an explicit computation of   the second zigzag cohomology group of a given 4-dimensional  Jacobi-Jordan algebra.  Moreover, we provide an example of deformation of the considered Jacobi-Jordan algebra and  show  that this deformation is isomorphic to another Jacobi-Jordan algebra which is  given in the 4-dimensional  classification, see  \cite{burde}.
\section{Preliminaries}\label{sec2}
In this section, we introduce some  notations  and  develop some necessary tools. For general results about Jacobi-Jordan algebras and their  applications see for instance \cite{amir2} and \cite{burde}.

In the sequel, we work over a field $\mathbb{K}$ of characteristic 0. We refer to an nonassociative algebra or simply an algebra as a pair $({\mathcal A},.)$ consisting of a $\mathbb{K}$-vector space $\mathcal A$ and a bilinear map ${\mathcal A}\otimes {\mathcal A}\rightarrow {\mathcal A};\,\, (x,y)\mapsto xy$.

\begin{definition}A Jacobi-Jordan algebra is a pair $({\mathfrak J},\cdot)$ consisting of a vector space $\mathfrak J$ and a bilinear map ${\mathfrak J}\otimes {\mathfrak J}\rightarrow {\mathfrak J};\,\, (x,y)\mapsto xy,$ satisfying the identities
\begin{eqnarray*}
	&(i)&  xy= yx,\\
	&(ii)& J(x,y,z):=x(yz)+y(zx)+z(xy)=0,\,\,   \forall\,  x,y,z \in {\mathfrak J}.
\end{eqnarray*}
	\end{definition}


	\begin{Example}\label{ex2}
Let  ${\mathfrak  J}_4=( \mathbb{K}^4,\cdot )$ be the Jacobi-Jordan algebra defined and denoted by $A_{1,4}$  in  \cite[Proposition 3.4]{burde}.  It is defined with respect to a basis $\{ e_1,e_2,e_3,e_4\} $ by 
$e_1 \cdot e_1=e_2,\;  e_1 \cdot e_3=e_4
$. In the sequel we refer to this Jacobi-Jordan algebra and its under underlying vector space by ${\mathfrak  J}_4$ and use concatenation for the product. 
This example will be used in the sequel to  illustrate the low degree  cohomology of  Jacobi-Jordan algebras and we refer to it by ${\mathfrak  J}_4$.
\end{Example}

	\begin{definition}
	Let ${\mathfrak J}$ be an algebra. We consider the new product $\{.~,.\}$ defined on the vector space ${\mathfrak J}$ by:
	$$\{x, y\}:= xy + yx, \,\,\, \forall\, x, y, z \in {\mathfrak  J}.$$
	
	The algebra ${\mathfrak J}$ is called  Jacobi-Jordan admissible algebra if $({\mathfrak  J},\{.~,.\})$ is a Jacobi-Jordan algebra. 
\end{definition}

	\begin{definition}
	Let $({\mathfrak J},.)$ be a nonassociative algebra
The anti-associator of this algebra is the trilinear map
		$\mbox{Aasso}: {\mathfrak J}\times {\mathfrak J}\times {\mathfrak J} \rightarrow {\mathfrak J}$
		defined by: $\mbox{Aasso}(x,y,z):= (x\cdot y)\cdot z+x\cdot (y\cdot z),\,\,\, \forall\, x, y, z \in {\mathfrak J}.$
		
	We say that $({\mathfrak J},.)$ is  an anti-associative algebra if $\mbox{Aasso}(x,y,z)= 0,\,\, \forall \, x, y, z \in {\mathfrak J}$.
\end{definition}
	
	\begin{proposition}
	Let $({\mathfrak J},\cdot )$ be an anti-associative algebra. Define a new  product $\circ$ in ${\mathfrak J}$  by $x\circ y:= x\cdot y + y\cdot x, \,\,\, \forall\, x, y, z \in {\mathfrak J}.$ Then  $({\mathfrak J},\circ)$ is a Jacobi-Jordan algebra. Hence anti-associative algebras are  admissible Jacobi-Jordan algebras.
\end{proposition}
Straightforward computation. 	
	
\begin{definition}	
	A nonassociative algebra $({\mathfrak J},\cdot )$ 
 is a Jordan algebra if it is commutative and if \begin{align}\label{jordan}   
x\cdot (x^2\cdot y)= x^2\cdot (x\cdot y), \, \forall x,y \in {\mathfrak J}.
 \end{align}
\end{definition}
\begin{remark}\label{jordan1}
It is shown in \cite{schafer} that  equation (\ref{jordan}) is equivalent to the following equation:
$$x((yz)t) + y((zx)t) + z((xy)t) = (xy)(zt) + (yz)(xt) + (zx)(yt),$$ for all $x,y,z\in {\mathfrak J}$.
\end{remark}
 For every elements $a_1,\dots,a_n$ of a commutative nonassociative algebra ${\mathfrak J}$, (where $n\in {\mathbb N}^*)$, and for $i \in \{1,\dots,n\},$ we define $a^i:= a_1\,$ if $i=1$ and $a^i:= a^{i-1}\cdot a_i\,$ if $i\geq 2.$ We note $a^i=: p_i(a_1,\dots,a_i).$  If $a_1= \dots = a_n= a$, $a^n:= p_n(a).$ An element $a$ of ${\mathfrak J}$ is called nilpotent if there exists $n\in {\mathbb N}^*$ such that $a^n= 0.$
For $n\in {\mathbb N}^*$,  ${\mathfrak J}^n$ is the set of all finite sums of product $p_n(a_1,\dots,a_n)$ of $n$ elements in  ${\mathfrak J}$. The algebra $({\mathfrak J},\cdot )$ is called nilpotent if there exists $n\in {\mathbb N}^*$ such that ${\mathfrak J}^n= \{0\}$.
If all elements of $({\mathfrak J},\cdot )$ are nilpotent, $({\mathfrak J},\cdot  )$ is called a nilalgebra.
Recall that A. A. Albert proved that all finite dimensional Jordan nilalgebras over a field of characteristic $\not= 2$ are nilpotent (see \cite{schafer}, Theorem 4.3 page 96).

Now, we recall an interesting lemma (Lemma 1 in \cite{RB}).

\begin{lemma}
Let $({\mathfrak J},\cdot )$ be a commutative 	algebra over a field of characteristic $\not= 2$. If $x^3= 0,\, \forall x \in {\mathfrak J},$ then $({\mathfrak J},\cdot )$ is a Jordan algebra.
\end{lemma}

Remark that if $({\mathfrak J},.)$ is a Jacobi-Jordan algebra, then $x^3= 0,\,\, \forall \, x\in {\mathfrak J}$. Using this lemma and the Albert's result above, the following proposition was  proved in \cite{amir2}:
\begin{proposition}\label{NIL} Every Jacobi-Jordan algebra $({\mathfrak J},.)$ is a nilpotent Jordan algebra such that $x^3= 0,\,\, \forall \, x\in {\mathfrak J}.$
\end{proposition}
In ($\cite{burde},$ Lemma $2.4$), we found a nice characterization of Jacobi-Jordan algebras: $({\mathfrak J},.)$ is a Jacobi-Jordan algebra if and only if $({\mathfrak J},.)$ is a commutative algebra 
 such that  $x^3= 0,\,\, \forall \, x\in {\mathfrak J}.$

Now, we shall define modules and representations of a Jacobi-Jordan algebra as in \cite{agore}. These definitions can be found also in \cite{jac} for Jordan algebras. 

\begin{definition} Let $\mathfrak J$ be a Jacobi-Jordan algebra. A (left) Jacobi-Jordan $\mathfrak J$-module (Jacobi-Jordan $\mathfrak J$-module,
for short) is a vector space $M$ equipped with a bilinear map $\star : \mathfrak J\times M \longrightarrow M $, called action,
such that for any $x,y \in  \mathfrak J$ and $a \in M$ :
\begin{align}\label{module}
    (xy) \star a = -x\star (y\star a)-y\star (x\star a).
    \end{align}
\end{definition}

\begin{definition}\label{def2}
	Let	$({\mathfrak  J},.)$ be a Jacobi-Jordan algebra and  $M$  be a vector space. A linear map 
		$\pi:{\mathfrak  J}\longrightarrow End(M)$  is said to be a representation of the Jacobi-Jordan algebra if
			\begin{equation} \label{eq-0}
		\begin{array}{rll}
		&&\pi(xy)=-\pi(x)\pi(y)-\pi(y)\pi(x),
		\forall x,y \in {\mathfrak  J}.
		\end{array}
		\end{equation}

\end{definition}	
	
	\begin{remark}
	If we view the Jacobi-Jordan algebra as a Jordan algebra, one may use the definition of representations introduced by Jacobson. It turns out that if we replace $\pi$ by $-\pi$ we get Jacobson's representations satisfying the identity 
		\begin{equation} 
		\begin{array}{rll}
		&&\pi(xy)=\pi(x)\pi(y)+\pi(y)\pi(x),
		\forall x,y \in {\mathfrak  J}.
		\end{array}
		\end{equation}
	Notice that representations were also considered in \cite{agore,amir2,MDV}.
	\end{remark}
	Representations of a Jacobi-Jordan algebra ${\mathfrak  J}$ and Jacobi-Jordan ${\mathfrak  J}$-modules are two different ways of
describing the same structure: more precisely, there exists an isomorphism of categories between the category of the ${\mathfrak  J}$-modules and the category of representations of ${\mathfrak  J}$.

 The one-to-one correspondence between Jacobi-Jordan ${\mathfrak  J}$-module structures $\star$ on $M$ and
representations $\pi$ of ${\mathfrak  J}$ on $M$ is given by the two-sided formula $\pi(x)(a) := x\star a$, for
all $x \in {\mathfrak  J}$ and $a\in M$ .
Furthermore, it was proved in 	\cite{amir2} that	
		 $(\pi,M)$ is a representation of ${\mathfrak  J}$ if and only if the vector space
		${\mathfrak  J}= {\mathfrak  J}\oplus M$ endowed with the following
		product:
		\begin{eqnarray*}
			(x+v)(y+w)=xy+\pi(x)w+\pi(y)v, \,  \, \, \forall x,y \in
		{\mathfrak  J}, v,w \in M,
		\end{eqnarray*}
		is a Jacobi-Jordan algebra.

\begin{definition}\label{deriv}
	Let ${\mathfrak  J}$ be  an algebra and let
	 $D\in \mbox{Hom}({\mathfrak  J},{\mathfrak  J})$.
	\begin{enumerate}
		\item  $D$ is called  derivation of the algebra  $({\mathfrak  J},\cdot )$ if
		
		$$D(xy)= D(x)y + xD(y),\,\,\, \forall\, x, y \in {\mathfrak  J}.$$

		\item  $D$ is called  anti-derivation of the algebra  $({\mathfrak  J},\cdot )$ if
		
		$$D(xy)= -D(x)y - xD(y),\,\,\, \forall\, x, y \in {\mathfrak  J}.$$
		
	\end{enumerate}

	Recall that $\mbox{Hom}({\mathfrak  J},{\mathfrak  J})$ endowed with the commutator is a Lie algebra denoted by ${\mathfrak g}l({\mathfrak  J}).$ Denote by $\mbox{Der}({\mathfrak  J})$ the vector space of derivations of  the algebra  $({\mathfrak  J},\cdot )$ and $\mbox{ADer}({\mathfrak  J})$ the vector space of anti-derivations of the algebra  $({\mathfrak  J},\cdot )$. 

\end{definition}
Recall the following result  \cite[Remark 2.2]{amir2}.

		\begin{proposition}\label{anticomm}
		
		Let ${\mathfrak  J}$ be a Jacobi-Jordan algebra.
\begin{enumerate}		
\item For all $x \in {\mathfrak  J}, \,  L_x:{\mathfrak  J}\rightarrow{\mathfrak  J}$, defined by $L_x(y):=xy,\,\,\ \forall y\in {\mathfrak  J},$ belongs to $\mbox{ADer}({\mathfrak  J}).$
				
\item  If $D, D'\in  \mbox{ADer}({\mathfrak  J})$,  the anti-commutator $\{D,D'\}:= DD'+D'D$, of $D$ and $D'$, is an anti-derivation of $({\mathfrak  J},.)$ if and only if $$\{D,D'\}(xy)= D(x)D'(y) + D'(x)D(y), \,\, \forall \, x,y \in {\mathfrak  J}.$$

\end{enumerate}			
\end{proposition}
\section{Zigzag Cohomology of Jacobi-Jordan algebras}\label{Sec3}
In this section, we develop a cohomology theory for Jacobi-Jordan algebras. The first problem is to define the differential operators and then define the first and second cohomology groups. After that we give  interpretations of these groups.
\subsection{Definition of a zigzag Cohomology}
We first single out the following definition:

\begin{definition}
Let $\mathfrak J$ and $M$ be two vector spaces. An $n$-linear map  $f:\underbrace{{\mathfrak J}\times{\mathfrak J}\ldots\times \mathfrak J}_{n\ \text{times}}\longrightarrow M,$ is said to be:
\begin{enumerate}
    \item [a)]
symmetric  if:
$${\displaystyle \;f(x_{\sigma (1)},\dots ,x_{\sigma (n)})=f(x_{1},\dots ,x_{n})}\mbox{ for all } \sigma \in  \mathfrak{S}_{n},$$
\item[b)] skew-symmetric if:
$${\displaystyle \;f(x_{\sigma (1)},\dots ,x_{\sigma (n)})=sign(\sigma)f(x_{1},\dots ,x_{n})}\mbox{ for all } \sigma \in  \mathfrak{S}_{n},$$
\end{enumerate}
$\mbox{ where } {\mathfrak S}_{n} \mbox{ is the group of  permutations of }\{1,...,n\}$.  For $n\in \mathbb N$, the set of $n$-linear maps is denoted by $L^n({\mathfrak  J},M)$. The set of symmetric (resp. skew-symmetric) $n$-linear maps is denoted by $S^n({\mathfrak  J},M)$ (resp. $A^n({\mathfrak  J},M))$. Moreover, we set $L^0({\mathfrak  J},M)=S^0({\mathfrak  J},M)=A^0({\mathfrak  J},M)=M$.
\end{definition}
Now, let ${\mathfrak  J}$ be a Jacobi-Jordan
algebra  over   a field $\mathbb K$ of characteristic zero. Let $M$ be a ${\mathfrak  J}$-module and $\pi:{\mathfrak  J}\longrightarrow End(M)$ be a representation of ${\mathfrak  J}$ in $M$.
Define the two following sequences of operators $(d^n)_n $ and $(\delta^n)_n$:
$$0 \longrightarrow {L}^{0}({\mathfrak  J},M)\xrightarrow[]{d^0} {L}^{1}({\mathfrak  J},M)\xrightarrow[]{d^1} {L}^{2}({\mathfrak  J},M)\xrightarrow[]{d^2}L^3({\mathfrak  J},M)\xrightarrow[]{d^3} ...L^n({\mathfrak  J},M)\xrightarrow[]{}\xrightarrow[]{d^n}L^{n+1}({\mathfrak  J},M)\xrightarrow[]{}...,$$

 $$\delta^n:A^n({\mathfrak  J},M)\longrightarrow L^{n+1}
({\mathfrak  J},M)$$\\

by
$$d^0(m)(x)=\delta^0(m)(x):=\pi(x)m,$$
\begin{align*}
    d^p(c)(x_1,...,x_{p+1})&=\sum_{i=1}^{p+1}\pi(x_i)c(x_1,...,\widehat{x_i},....,x_{n+1})+\sum_{1\leq i<j\leq p+1 }c(x_ix_j,x_1,...,\widehat{x_i},...,\widehat{x_j},...,x_{n+1}),
\end{align*}
and 
\begin{align*}
    \delta^p(c)(x_1,...,x_{p+1})&=\sum_{i=1}^{p+1}\pi(x_i)c(x_1,...,\widehat{x_i},....,x_{n+1})-\sum_{1\leq i<j\leq p+1 }c(x_ix_j,x_1,...,\widehat{x_i},...,\widehat{x_j},...,x_{n+1}),
\end{align*}
\begin{theorem}\label{prop6}
 For $p\geq 1$, we have $d^p\circ \delta^{p-1}=0$.
\end{theorem}
\begin{proof}
Let $c$ be a skew-symmetric  $(p-1)$-linear map in $A^{p-1}({\mathfrak  J},M)$.
We have,
\begin{align*}
&d^p\circ \delta^{p-1}c(x_1,...,x_{p+1})=\\& \sum_{i=1}^{p+1}~\pi(x_i)\delta^{p-1}c(x_1,...,\widehat{x_i},...x_{p+1})+ \sum_{1\leq i<j\leq p+1}\delta^{p-1}c(x_ix_j,x_1,...,\widehat{x_i},...,\widehat{x_j},...,x_{p+1}).
\end{align*}
The first sum of the right hand side:
\begin{align*}
& \sum_{i=1}^{p+1}~\pi(x_i)\delta^{p-1}c(x_1,...,\widehat{x_i},...x_{p+1})
=\\ &
\sum_{1\leq i< j\leq p+1}~\Bigl(\pi(x_i)\pi(x_j)+\pi(x_j)\pi(x_i)\Bigl)c(x_1,...,\widehat{x_i},...,\widehat{x_j},...,x_{p+1})\\&
-\sum \limits_{\underset{ j\neq i ,k\neq i, j<k}{1\leq i,j,k\leq p+1}}
~\pi(x_i)c(x_j x_k,...,\widehat{x_i},...,\widehat{x_j},...,...,\widehat{x_k},...,x_{p+1}).
\end{align*}
The second sum of the right hand  side:
\begin{align*}
&\sum_{1\leq i<j\leq p+1}\delta^{p-1}c(x_ix_j,x_1,...,\widehat{x_i},...,\widehat{x_j},...,x_{p+1})=
\\ &\sum_{1\leq i< j\leq p+1}~\pi(x_ix_j)c(x_1,...,\widehat{x_i},...,\widehat{x_j},...,x_{p+1})+\sum \limits_{\underset{ i\neq j,i\neq k, j<k}{1\leq i,j,k\leq p+1}}
~\pi(x_i)c(x_jx_k,x_1,...,\widehat{x_i},...,\widehat{x_j},...,\widehat{x_k},...,x_{p+1})\\
& -\sum \limits_{\underset{ \{i,j\}\cap
\{k,l\}=\emptyset,i<j,k<l }{1\leq i,j,k,l\leq p+1}}
~c(x_ix_j,x_kx_l,x_1...,\widehat{x_i},...,\widehat{x_j},...,\widehat{x_k},...,\widehat{x_l},...,x_{p+1})\\
&-\sum \limits_{\underset{ i<j,k\neq i, k\neq j }{ 1\leq i,j, k \leq  p+1}}
~c((x_ix_j)x_k,x_1...,\widehat{x_i},...,\widehat{x_j},...,\widehat{x_k},...,x_{p+1})\\&=\sum_{1\leq i< j\leq p+1}~\pi(x_ix_j)c(x_1,...,\widehat{x_i},...,\widehat{x_j},...,x_{p+1})+\sum \limits_{\underset{ j,k\neq i, j<k}{1\leq i,j,k\leq p+1}}
~\pi(x_i)c(x_jx_k,x_1...,\widehat{x_i},...,\widehat{x_j},...,\widehat{x_k},...,x_{p+1}),
\end{align*}
since $c$ is skew-symmetric and by using  Jacobi identity.
Consequently,
\begin{align*}
&d^p\circ \delta^{p-1}c(x_1,...,x_{p+1})=\\&
\sum_{1\leq i< j\leq p+1}~\Bigl(\pi(x_i)\pi(x_j)+\pi(x_j)\pi(x_i)+\pi(x_ix_j)\Bigl)c(x_1,...,\widehat{x_i},...,\widehat{x_j},...,x_{p+1})=0,
\end{align*}
because $\pi$ is a representation.

\end{proof}

\begin{notation}
We denote by $ ~_zZ^p
({\mathfrak  J}, M) :=\ker ( d^p)$
the space of $p$-cocycles, and by
$~_zB^p({\mathfrak  J},M) := \{c \in L^p({\mathfrak  J},M) ; \exists c' \in A^{p-1}({\mathfrak  J},M) : c = \delta^{p-1}c' \}$
the space of $p$-coboundaries. Then we define the $p^{\rm th}$ cohomology space of $\mathfrak  J$ with values/coefficients in $ M$
as the quotient 
$$~_zH^p(\mathfrak  J,M) := ~_zZ^p(\mathfrak  J,M) /~_zB^p(\mathfrak  J,M).$$
\end{notation}

\subsection{Low degree cohomology spaces and Extensions}
 An algebraic motivation
is the interpretation of the low degree cohomology spaces in terms of algebraic properties of the Jacobi-Jordan
algebra, which makes the cohomology spaces interesting invariant to compute. In degree zero we have:

$$~_zH^0({\mathfrak  J},M) = ~_zZ^0({\mathfrak  J},M) = \{m \in M ; d^0m = 0\} = \{m \in  M ; \forall x \in {\mathfrak  J} ; \pi(x)m = 0\}.$$
A particular case is here $M = {\mathfrak  J}$ with the adjoint action.
The quotient space is then
$$~_zH^0({\mathfrak  J}, {\mathfrak  J}) = Ann({\mathfrak  J}),$$ the annihilator of ${\mathfrak  J}$.

Now, we shall investigate the first group of cohomology.
Let $~_zZ^1({\mathfrak  J}, M)$ be the module of all linear maps $f$ from ${\mathfrak  J}$ to $M$
 such that
 $$\pi(x)f(y)+\pi(y)f(x)+f(xy)=0,\,\,\, x,y\in {\mathfrak  J}$$
for all $x,y\in {\mathfrak  J}$. Let $~_zB^1({\mathfrak  J}, M)$ be all those $f$ such that there is  an element $m$ belonging to $ M$ with $f(x)=\pi(x)m$ for all $x\in{\mathfrak  J}$. 
$$~_zH^1({\mathfrak  J},M) = ~_zZ^1({\mathfrak  J},M) /~_zB^1({\mathfrak  J},M) = ADer({\mathfrak  J},M) / PDer({\mathfrak  J},M),$$
with
$$ADer({\mathfrak  J},M) := \{f \in Hom({\mathfrak  J},M) ; \forall x, y \in {\mathfrak  J} : f(xy) = -\pi(x)  f(y) - \pi(y ) f(x)\},$$
and
$$PDer({\mathfrak  J},M) := \{f \in Hom({\mathfrak  J},M) ;  \exists m \in M ; \forall x \in {\mathfrak  J} : f(x) = \pi(x)m\}.$$ 
One of the most important case is where $M = {\mathfrak  J}$ with the adjoint action $(L)$ defined by 
$L_x(y)=xy,~~x,y\in\mathfrak  J$.
In this case, we use the following notation
$ADer({\mathfrak  J}, {\mathfrak  J}) = ADer({\mathfrak  J}) $, defining 
 the linear space of (all) antiderivations of ${\mathfrak  J}$ , and  $IADer({\mathfrak  J}, {\mathfrak  J}):=IADer({\mathfrak  J})$,  the subspace of inner derivations, i.e.
$IADer({\mathfrak  J}):=\{L_x; ~~ x\in {\mathfrak  J}\}.$ Indeed,  Proposition \ref{anticomm} says that for all $x$ in ${\mathfrak  J}$, $L_x:{\mathfrak  J}\longrightarrow {\mathfrak  J}$ defined by $L_x(y)=xy$, $\forall y\in {\mathfrak  J}$ is an antiderivation.
The quotient space is then
$~_zH^1({\mathfrak  J}, {\mathfrak  J}) = OADer({\mathfrak  J}),$
the space of outer antiderivations of ${\mathfrak  J}.$\\ 

Another very important case is $M = \mathbb K$ with the trivial action. Here
$$~_zH^1({\mathfrak  J}, \mathbb K) = \{f \in Hom({\mathfrak  J}, \mathbb K) ; \forall x, y \in {\mathfrak  J} : f(xy) = 0\} / \{0\} = ({\mathfrak  J} / {\mathfrak  J}^2)^* .$$

\begin{Example}[Computation of $_zH^1({\mathfrak J}_4, \mathbb K)$]

	
	Let $h$ be a linear form of ${\mathfrak J}_4$. By easy computation we can prove that  $h$ belongs to $~_zH^1({\mathfrak J}_4, \mathbb K) = \{f \in Hom({\mathfrak  J}_4, \mathbb K) ; \forall x, y \in {\mathfrak J}_4 : f(xy) = 0\}$
	 if and only if the matrix of $h$ with respect to  the basis $\{e_1,e_2,e_3,e_4\}$ is of the form $\left(
  \begin{array}{cccc}
    a & 0 & b & 0 
  \end{array}
\right),\,\, a,b\in \mathbb K.$ Thus, $~_zH^1({\mathfrak  J}_4, \mathbb K)$ is two dimensional. Moreover, $~_zH^1({\mathfrak  J}_4, \mathbb K)$ is spanned by $\{h_1:=\left(
  \begin{array}{cccc}
    1 & 0 & 0 & 0 
  \end{array}
\right),\,\,h_2:=\left(
  \begin{array}{cccc}
    0 & 0 & 1 & 0 
  \end{array}
\right)\}.$
\end{Example}
\begin{Example}[Computation of $_zH^1({\mathfrak J}_4, {\mathfrak J}_4)$]
	 
	 We next investigate the first group of cohomology of  ${\mathfrak  J}_4$ with the adjoint representation.
   We find the dimension and generators of $~_zH^1({\mathfrak  J}_4,{\mathfrak  J}_4)$.
   
	An endomorphism  $D$  on the previous algebra ${\mathfrak  J}_4$ is an antiderivation  if and only if the matrix of $D$ on the basis  $\{e_1,e_2,e_3,e_4\}$ is given by
$$\left(
  \begin{array}{cccc}
    \alpha_1 & 0 & 0 & 0 \\
    \alpha_2 & -2\alpha_1 & \alpha_3 & 0 \\
    \alpha_4 & 0 & \alpha_5 & 0 \\
    \alpha_6 & \alpha_4 & \alpha_{7} & -\alpha_{1}-\alpha_{5} \\
  \end{array}
\right),$$
where, $(\alpha_i)_{\{1\geq i\geq 7\}}$ are scalar.

The matrix of an inner antiderivation $L_x:{\mathfrak  J}_4\longrightarrow{\mathfrak  J}_4$ defined by $L_x(y):=xy;\,\,\, \forall y\in {\mathfrak  J}_4$ has the following form:
$$\left(
  \begin{array}{cccc}
    0 & 0 & 0 & 0 \\
    a & 0 & 0 & 0 \\
    0 & 0 & 0 & 0 \\
    b & 0 & a & 0 \\
  \end{array}
\right),$$
where $a$ and $b$ are scalar.
Thus, the dimension of $~_zH^1({\mathfrak  J}_4,{\mathfrak  J}_4)$ is $5$ and elements of $~_zH^1({\mathfrak  J}_4,\mathbb K)$ can be written as the classes of the following matrices:
$$\left(
  \begin{array}{cccc}
    \alpha_1 & 0 & 0 & 0 \\
    \alpha_2 & -2\alpha_1 & \alpha_3 & 0 \\
    \alpha_4 & 0 & \alpha_5 & 0 \\
    0 & \alpha_4 & 0 & -\alpha_{1}-\alpha_{5} \\
  \end{array}
\right),$$
where, $(\alpha_i)_{\{1\geq i\geq 5\}}$ are scalar.


\end{Example}


Now, we deal with the relationship between equivalence classes of abelian extensions and the second cohomology grous.

\begin{definition}\label{def1}
A short exact sequence of Jacobi-Jordan algebras 
\begin{eqnarray}\label{exact}
0 \longrightarrow M\xrightarrow[]{i}E\xrightarrow[]{p} {\mathfrak  J}\longrightarrow 0 ,
\end{eqnarray}
 is called an abelian
extension of $\mathfrak  J$  by $M$ in case $(i(M))^2=\{0\}$, where $i : M\longrightarrow E$ is the inclusion
map and $p : E\longrightarrow {\mathfrak  J}$ the projection.
As a vector space, $E=M\oplus {\mathfrak  J}$ and the subspace $M$ is obviously an $\mathfrak  J$-module, the Jacobi-Jordan algebra
structure on $E$
being given by:
\begin{eqnarray*}
(a\oplus x)(b\oplus y) = (\pi(x) b +\pi( y) a + c(x, y)\oplus x y).
\end{eqnarray*}
where $a,b\in M$, $x,y\in{\mathfrak  J}$ and  $c :{\mathfrak  J}\times{\mathfrak  J}\longrightarrow M$ is a bilinear map. Such a Jacobi-Jordan algebra is denoted by $E_c$.
\end{definition}
\begin{definition}
\begin{enumerate}\label{equivalentext}
\item[(1)] Two  abelian extensions $E_c$ and $E_{c'}$ of $\mathfrak  J$ by $M$ are called
equivalent in case the following diagram is commutative:
\[ \begin{tikzcd}\label{diag1}
0 \arrow{r}{}&M \arrow{r}{i} \arrow[swap]{d}{id_M} & E_c \arrow{r}{p}\arrow{d}{\varphi} &{\mathfrak  J}\arrow{r}{}\arrow[swap]{d}{id_{\mathfrak  J}}&0\\%
0 \arrow{r}{}&
M \arrow{r}{i}& E_{c'}\arrow{r}{p}& {\mathfrak  J} \arrow{r}{}& 0
\end{tikzcd}
\]
In this case, $\varphi$ is necessarily an isomorphism. Denote by $Ext({\mathfrak  J},M)$ the set (actually an abelian group)
of equivalence classes of abelian extensions of $\mathfrak  J$ by $ M$.
\item[(2)]An extension (\ref{exact}) is called trivial if it is equivalent to  the extension (\ref{semidirect}) with c = 0,  that is   the Jacobi-Jordan algebra $E_0=M\rtimes {\mathfrak  J}$
is  the semi-direct product of $\mathfrak J$ by $M$. If such an isomorphism does not exist, then the extension is called non-trivial.
\item[(3)] An extension (\ref{exact}) of a Jacobi-Jordan algebra is called  central extension if the module $M$ is trivial,
i.e., if the $\mathfrak  J$-action on $M$ is identically zero. In this case, $M$ belongs to the annihilator of $E$.
\end{enumerate}
\end{definition}

\begin{notation}
Hereafter we denote by:
$$~_zZ_a^n({\mathfrak  J},M)=~_zZ^n({\mathfrak  J},M)\bigcap A^n(\mathfrak  J,M).$$
$$~_zZ_s^n({\mathfrak  J},M)=~_zZ^n({\mathfrak  J},M)\bigcap S^n(\mathfrak  J,M).$$
$$ ~_zH_s^2({\mathfrak  J},M):= ( ~_zZ_s^2({\mathfrak  J},M)/~_zB^2({\mathfrak  J},M)\Bigl).$$
$$ ~_zH_s^3({\mathfrak  J},M):= ( ~_zZ_s^3({\mathfrak  J},M)/~_zB^3({\mathfrak  J},M)\Bigl).$$
\end{notation}
\begin{proposition}
 Let ${\mathfrak  J}$ be a Jacobi-Jordan algebra and $(\pi,M)$ be a representation of ${\mathfrak  J}$. Consider $$ C(M)=\{m\in M; \pi(x)m=0,~\forall x\in\mathfrak J\}.$$
 If $ C(M)=\{0\}$, then $ ~_zH^2({\mathfrak  J},M)= ~_zH_s^2({\mathfrak  J},M)$.
\end{proposition}
\begin{proof}
 Let $c\in  ~_zZ^2({\mathfrak  J},M)$. Then, there exist $c_s\in S^2(\mathfrak J)$ and  $c_a\in A^2(\mathfrak J)$ such that $c=C_s+c_a$. We have $d^2c(x,y,z)-d^2c(y,x,z)=2\pi(z)c_a(x,y).$
 Thus,  $\pi(z)c_a(x,y)=0$ for all $z\in \mathfrak J$, because $c$ is a $2$-cocycle. Consequently, $c_a(x,y)$ belongs to $C(M)=\{0\}$.
\end{proof}

\begin{theorem}\label{equivext}
$$~_zH_s^2(\mathfrak  J,M) \simeq
Ext(\mathfrak  J,M).$$
\end{theorem}
\begin{proof}
Suppose that $E=\mathfrak J\oplus M$ endowed with  the following product:
\begin{eqnarray}\label{semidirect}
(a\oplus x)(b\oplus y) = (\pi(x) b +\pi( y) a + c(x, y)\oplus x y),
\end{eqnarray}
where $a,b\in M$, $x,y\in{\mathfrak  J}$ and  $c :{\mathfrak  J}\times{\mathfrak  J}\longrightarrow M$ is a bilinear map,
is a Jacobi Jordan algebra. It is clear that product (\ref{semidirect}) is commutative if and only if $c$ is symmetric. Moreover, the Jacobi identity is equivalent to the fact  that $c$ is a $2$-cocycle. Now, we shall prove that two extensions $E_c$ and $E_{c'}$ are equivalent   if and only if there exists  an isomorphism  $\varphi $ defined by
$$\varphi (a\oplus x):= a+ \psi(x)\oplus x ,$$ 
where $\psi :{\mathfrak  J}\longrightarrow M$  is a linear map.  In fact, Since $p\circ\varphi(0\oplus x)=p(0\oplus x)=x$, then there exists an element $\psi(x)\in M$ such that $\varphi(0\oplus x)=\psi(x)\oplus x$. On the other hand, $\varphi\circ i(a)=i(a)$. Then,  $\varphi (a\oplus 0)=a\oplus 0$. This means that
$$\varphi(a\oplus x)=\varphi(a\oplus 0)+\varphi(0\oplus x)=a+ \psi(x)\oplus x.$$ One can easily see that any such morphism is bijective. Consequently, $E_c$ and $E_{c'}$ are equivalent   if and only if $\varphi $ is an isomorphism, which  is equivalent to the fact that $c'=c+\delta^2(\psi)$. That is $c$ and $c'$  are cohomologous.
\end{proof}
\begin{corollary}
The set of trivial abelian extensions is in one to one correspondence with  $~_zB^2(\mathfrak  J,M).$
\end{corollary}
\begin{proof}
Let $E:{\mathfrak  J}\oplus M$ be a trivial abelian extension of $\mathfrak  J$. The product on $E$ is defined as above by:
$
(a\oplus x). (b\oplus y) = \pi(x) b+ \pi(y) a + c(x, y)\oplus x y.
$

Since $E$ is isomorphic to
 the semi-direct product  $M\rtimes {\mathfrak  J}$. By the proof of Theorem \ref{equivext} the isomorphism is defined above by:  $\varphi:a\oplus x \mapsto a+ \psi(x)\oplus x .$
 An easy computation implies that this fact is equivalent to the fact $c=\delta\psi.$
\end{proof}


   In the following example,  We compute the dimension and generators of $~_zH_s^2({\mathfrak  J}_4,\mathbb K).$ This characterises the abelian extensions of ${\mathfrak  J}_4$ by $\mathbb K$ called the one dimensional central extensions.

   \begin{Example}[Computation of $_zH^2({\mathfrak J}_4, \mathbb K)$]
  
	Let $c$ be an element of $~_zZ^2({\mathfrak  J}_4,\mathbb K).$ Then $c$ is a  bilinear form of ${\mathfrak  J}_4$ which satisfies the following equation:
	\begin{align}\label{s7}
	    c(xy,z)+c(yz,x)+c(xz,y)=0, \text{ for all } x,y,z\in{\mathfrak  J}_4.
	    	\end{align}
	    Apply the equation above to $x=e_i,\,y=e_j\,z=e_k$, we get the following assertions:
	$$
\begin{cases} \text{ For } i=j=k=1,\quad  c(e_2,e_1)=0,\\ 
\text{ For } i=j=1 \text{ and } k=2,\quad   c(e_2,e_2)=0,\\ 
\text{ For } i=j=1 \text{ and } k=3,\quad  c(e_2,e_3)=-2c(e_4,e_1),\\
\text{ For } i=j=1 \text{ and } k=4, \quad   c(e_2,e_4)=0,\\
\text{ For } i=1 \text{ and } j=k=3, \quad   c(e_4,e_3)=0,\\
\text{ For } i=1,\,j=3 \text{ and } k=4, \quad   c(e_4,e_4)=0.
\end{cases}
$$
Thus, the matrix of $c$ with respect to the basis $(e_1,e_2,e_3,e_4)$ is given by:
$$\left(
  \begin{array}{cccc}
    \alpha_{11} & 0 & \alpha_{13} & \alpha_{14} \\
    0 & 0 & -\frac{1}{2}\alpha_{14} & 0 \\
    \alpha_{31} & \alpha_{32} & \alpha_{33} & \alpha_{34} \\
    \alpha_{41} & 0 & 0 & 0 \\
  \end{array}
\right).$$
Consequently,  $~_zZ^2({\mathfrak  J}_4,\mathbb K)$ is 8-dimensional and generated by $\{E_{11},E_{13},E_{31},E_{32},E_{33},E_{34},E_{41},E_{14}-\frac{1}{2} E_{23} \}$, where $(E_{ij})$ are elementary matrices. 
\end{Example}

\begin{proposition}\label{prop5}
 The space  $~_zH^2({\mathfrak  J}_4,\mathbb K)$ is $6$-dimensional  and   generated by  the classes
 $$\{E_{13},E_{32},E_{33},E_{34},E_{41},E_{14}-\frac{1}{2} E_{23} \}.$$
\end{proposition}

\begin{proof}

Let $f$  be the  linear form of ${\mathfrak  J}_4$ defined  by the matrix $\left(
  \begin{array}{cccc}
    a & b & c & d \\
  \end{array}
\right)$,  
  with respect to the  basis $(e_1,e_2,e_3,e_4)$. We have $\delta f(e_1,e_1)=f(e_2)=b$, $\delta f(e_1,e_3)=\delta f(e_3,e_1)=f(e_4)=d$ and $\delta f(e_i,e_j)=0$ if $(i,j)\notin \{(1,1),(1,3),(3,1)\}$.
  Thus the matrice $M( \delta f)$ of $\delta f$ is given by 
  $M( \delta f)=bE_{11}+d(E_{13}+E_{31}).$ Thus, $dim ~_zB^2({\mathfrak  J}_4,\mathbb K):=dim(Im\delta^1)=2$ and $~_zB^2({\mathfrak  J}_4,\mathbb K)$ is generated by $ E_{11}$ and $E_{13}+E_{31}.$ Consequently,  $dim ~_zH^2({\mathfrak  J}_4,\mathbb K)=6$ and $ ~_zH^2({\mathfrak  J}_4,\mathbb K)$ is generated by $\{E_{13},E_{32},E_{33},E_{34},E_{41},E_{14}-\frac{1}{2} E_{23} \},$  because $dim ~_zZ^2({\mathfrak  J}_4,\mathbb K)=8$ and $~_zZ^2({\mathfrak  J}_4,\mathbb K)$ is generated by $\{E_{11},E_{13},E_{31},E_{32},E_{33},E_{34},E_{41},E_{14}-\frac{1}{2} E_{23} \}$.
\end{proof}

   \begin{corollary}
    The space $Ext({\mathfrak J}_4,\mathbb K)$ is 2-dimensional and generated by $$\{ E_{33}, E_{23}+E_{32}-2E_{14}-2 E_{41} \}.$$
   \end{corollary}
   \begin{proof}
    We have from Example \ref{ex2} that the dimension of $~_zZ^2({\mathfrak  J}_4,\mathbb K)$ is $8$ and  generated by $\{E_{11},E_{13},E_{31},E_{32},E_{33},E_{34},E_{41},E_{14}-\frac{1}{2} E_{23} \}.$ Thus $~_zZ_s^2({\mathfrak  J}_4,\mathbb K)$  
   	 is of dimension $4$ and generated by the cocycles defined by $\{c_1=E_{11},c_2=E_{13}+E_{31},c_3=E_{33},c_4=E_{23}+E_{32}-2E_{14}-2 E_{41} \}.$

We have shown that $c_1$ and $c_2$ are  coboundaries.
Moreover, we have proved in the previous proposition that the classes of  $E_{32},E_{33},E_{34},E_{41},E_{14}-\frac{1}{2} E_{23} ,$ are linearly independent. Thus, 
$c_3 $ and $c_4 $ are not  coboundaries. Consequently, the dimension of $~_zH_s^2({\mathfrak  J}_4,\mathbb K)$ is $2$ and  generators are  $\{c_3,c_4\}$.  
\end{proof}

\section{Scalar  zigzag cohomology of pseudo-Euclidean Jacobi Jordan algebras}\label{Sec4}
We aim in this section to consider Jacobi-Jordan algebras with nondegenerate, symmetric bilinear forms and discuss their zigzag cohomology.
	\begin{definition}
		Let $({\mathfrak  J},\cdot)$ be a Jacobi-Jordan algebra. A bilinear form $B$ on  ${\mathfrak  J}$ is said to be an associative  scalar product on  $({\mathfrak  J},\cdot)$ if $B$ is nondegenerate, symmetric and associative bilinear form. Recall that $B$ is  associative means $B(x\cdot y,z)=B(x,y\cdot z), $ for all $ x,y,z \in {\mathfrak  J}.$
		
		A Jacobi-Jordan algebra  $({\mathfrak  J},\cdot )$ is called a pseudo-Euclidean Jacobi-Jordan algebra if it is endowed with an associative scalar product. \\
		We denote by $F(\mathfrak J)$ the vector space of all invariant symmetric bilinear forms of $\mathfrak J$ and by $\Delta_p$ the dimension of $F(\mathfrak J)$. The vector subspace  of $F(\mathfrak J)$ generated by all nondegenerate invariant symmetric bilinear forms  is denoted by $S(\mathfrak J)$ and its dimension by $b_p$. For more details about this type of algebras see \cite{ami} and \cite{amir5}.
	\end{definition}
	\begin{remark}
	Recall that in   \cite[Lemma 7.1.1 page 103]{amir5}, it is proved that if the field $\mathbb K$ is algebraically closed then $F(\mathfrak J)$ is equal to the vector space spanned by all associative scalar product of $\mathfrak J$, that is $\Delta_p=b_p$.
	\end{remark}
	\begin{definition}
		Let $({\mathfrak  J},B)$ be a pseudo-Euclidean Jacobi-Jordan algebra. A (anti-)derivation  $D$ is said to be symmetric  (resp. skew-symmetric) with respect to $B$ if  $B(D(x),y)=B(x,D(y))$ (resp.  $B(D(x),y)=-B(x,D(y))$) for all $x,y\in \mathfrak J$.
		
		The set of all symmetric (resp. skew-symmetric) derivations with respect to $B$  is denoted by $Der_s({\mathfrak  J})$ (resp. $Der_a({\mathfrak  J})$). We denote by $\Delta_s$ (resp. $\Delta_a$) the dimension of  $Der_s({\mathfrak  J})$ (resp.  $Der_a({\mathfrak  J})$). We also denote by $\Delta$ the dimension of $Der({\mathfrak  J})$.
	\end{definition}
	\begin{proposition}\label{transpose}
Let $(\mathfrak J,B)$ be a pseudo-Euclidean Jacobi-Jordan algebra. We have

 $$ ~_zZ^2(\mathfrak J,\mathbb K)\simeq  \{h\in End(\mathfrak J):  h(xy)=-x~^th(y)-y~^th(x)\}.$$
\end{proposition}
\begin{proof}
 Let $c\in ~_zZ^2(\mathfrak J,\mathbb K)$. There exist  $c_s\in S^2(\mathfrak J)$ and  $c_a\in A^2(\mathfrak J)$ such that $c=c_s+c_a$. The fact that $B$ is symmetric and nondegenerate implies that there exist $f,g\in End(\mathfrak J)$ such that 
 $c_s(x,y)=B(f(x),y)$ and $c_a(x,y)=B(g(x),y)$. It is clear that $f=~^tf$ and $g=-~^tg$. By the invariance of $B$ we get the result for $h=f+g$. The converse is clear.
\end{proof}
	\begin{corollary}\label{corol3}
Let $(\mathfrak J,B)$ be a pseudo-Euclidean Jacobi-Jordan algebra. We have
\begin{enumerate}
    \item[(i)]
 $ ~_zH_s^2(\mathfrak J,\mathbb K)\simeq  ADer_s(\mathfrak J)/IADer(\mathfrak J),$
 \item[(ii)] $~_zZ_a^2(\mathfrak J,\mathbb K)\simeq Der_a(\mathfrak J).$
 \end{enumerate}
\end{corollary}
\begin{proof}
 (i) Let $c\in~_zZ_s^2(\mathfrak J,\mathbb K)$. By Proposition \ref{transpose}, there   exists $h\in End(\mathfrak J)$ such that 
$ c(x,y)=B(h(x),y)$ for all $x,y\in \mathfrak J$. The fact that $c$ is symmetric implies that $h$ is $B$-symmetric, that is $h$ belongs to  $ADer_s(\mathfrak J)$.
Now, if $c\in ~_zB^2(\mathfrak J,\mathbb K)$, then there exist a linear map $f$ and an element $a_0$ of $\mathfrak J$ such that $c(x,y)=\delta^1 f(x,y)=B(a_0x,y)$ for all $x,y\in \mathfrak J.$ Since $B$ is nondegenerate, then, $\delta^1 f=L_{a_0}$. 

(ii) One gets it in a  similar way.
\end{proof}
\begin{lemma}
Let $\mathfrak J$ be a pseudo-Euclidean Jordan algebra and $\varphi$ be a symmetric
invariant bilinear form on $\mathfrak J$ (i.e., $ \varphi \in F(\mathfrak J))$. Then, the map $\mu(\varphi)
: \mathfrak J\times \mathfrak J\times \mathfrak J\longrightarrow\mathbb K$, defined
by $$\mu( \varphi) (x,y,z)= \varphi (xy,z),\text{ for all } x,y,z \in \mathfrak J,$$ is an element of $~_zZ_s^3(\mathfrak J,\mathbb K)$.
\end{lemma}
\begin{proof}
Since $\varphi$ is symmetric and invariant, then $\mu(\varphi)$ is symmetric. Furthermore, it follows: 
\begin{align*}
   & d^3\mu(\varphi)(x,y,z,t)\\&=\mu(\varphi)(xy,z,t)+\mu(\varphi)(xz,y,t)+\mu(\varphi)(xt,y,z)+\mu(\varphi)(yz,x,t)+\mu(\varphi)(yt,x,z)+\mu(\varphi)(zt,x,y)\\& =\varphi((xy)z+(zx)y+(yz)x,t)+\varphi((xt)y+(ty)x+(yx)t,z)=0,
    \end{align*}
    for all  $x,y,x,t\in \mathfrak J.$
     We conclude that $\mu(\varphi)$ is an element of $~_zZ_s^3(\mathfrak J,\mathbb K)$.
\end{proof}
\begin{lemma}
 Let $(\mathfrak J,B)$ be a pseudo-Euclidean Jacobi-Jordan algebra. If $D \in Der(\mathfrak J)$, then the map $\nu(D): \mathfrak J\times\mathfrak J\longrightarrow\mathbb K$, defined by $$\nu(D)(x,y ) =
B(D(x), y) + B(x, D(y )), \text{ for all } x,y\in \mathfrak J,$$ is an element of $F(\mathfrak J).$
\end{lemma}
\begin{proof}
  Since B is a symmetric bilinear form and $D$ is a linear form, it follows
that $\nu(D)$ is a symmetric bilinear form.  Moreover, $\nu(D)$ is invariant because B is invariant and $D$ is a derivation.
\end{proof}
\begin{theorem}\label{h3} Let $(\mathfrak J,B)$ be a pseudo-Euclidean Jacobi-Jordan algebra.
The following sequence 
\begin{eqnarray}
0 \longrightarrow Der_a({\mathfrak  J},B)\xrightarrow[]{i}Der({\mathfrak  J})\xrightarrow[]{\nu}F({\mathfrak  J}) \xrightarrow[]{\bar\mu} ~_zH_s^3(\mathfrak J,\mathbb K) 
\end{eqnarray}
is exact, where
$i(D)=D,\text{ for all } D\in Der(\mathfrak J)$ and $\bar\mu(\varphi)$ is the cohomology class of $\mu(\varphi)$.
\end{theorem}
\begin{proof}
  It is clear that $i(Der_a(\mathfrak J, B)) = Der_a((\mathfrak J, B)) = \ker(\nu).$ Now we are going to show that $\nu(Der(\mathfrak J)) = \ker(\bar\mu).$ Let $D \in Der(\mathfrak J);$ denote by $^t
D$ the transpose of $D$ with respect to $B$. 
Let $x,y,z\in \mathfrak J$,
\begin{align*}
    &\mu(\nu(D))(x,y,z)=\nu(D)(xy,z)=B(D(xy),z)+B(xy,D(z))\\&=B(D(x)y,z)+B(xD(y),z)+B(xy,D(z))\\&=B(xy,D(z))+B(yz,D(x))+B(zx,D(y))\\
    & =\varphi(xy,z)+  \varphi(yz,x)+\varphi(zx,y),
\end{align*}
where $\varphi: \mathfrak J\times\mathfrak J\longrightarrow\mathbb K$ is defined by $\varphi(x,y)=B((D-^tD)x,y)$. \\
Thus, $\mu(\nu(D))=\delta^2\varphi.$ That is, $\nu(D)$ belongs to $\ker\bar\mu$.
Conversely, let $T\in F(\mathfrak J)$ such that $\bar\mu(T)=0$, that is there exists a skew-symmetric bilinear form $\varphi$ on $\mathfrak J$ such that
$$\mu(T)(x,y,z):=T(xy,z)=\delta^2\varphi(x,y,z).$$
  The fact that $B$ is nondegenerate implies that there exist two endomorphisms $f,h$  of $\mathfrak J$ such that $T (x,y)=B(f(x),y)$ and $\varphi(x,y) = B(h(x),y), \forall x,y \in \mathfrak J.$ Since $T$ is  invariant, then  $f( xy) =f(x)y=xf(y), \forall x,y\in\mathfrak J$. Moreover, $f = ~^tf$,  where $^t f$ is the transpose of $f$ with respect to $B$ because $B$ and $T$ are symmetric. The fact that $\varphi$ is skew-symmetric implies that $ h =-^th$,  where $^th$ is the transpose of $h$ with respect to $B$. Let $x,y,z\in \mathfrak J$; since
  $$T( xy,z)=\varphi(xy,z)+ \varphi(yz,x)+\varphi(xz,y)$$ then $$B(f(xy),z)=B(h(xy)- xh(y)-yh(x),z).$$ Because $B$ is nondegenerate, then $$f(xy)=h(xy)-xh(y)-yh(x).$$  Consequently, $\frac{1}{2}(f +h)$ is a derivation of $\mathfrak J$. It is clear that $f =(\frac{1}{2}(f +h))+~^t(\frac{1}{2}(f +h))$, thus $ T =\nu(\frac{1}{2}(f +h)).$ We conclude that $Ker(\bar\mu)\subset\nu(Der(\mathfrak J));$ so $Ker(\bar\mu)=\nu(Der(g))$.

\end{proof}

The theorem leads to the following corollary.
\begin{corollary}\label{dp}
Let $(\mathfrak J,B)$ be a pseudo-Euclidean Jacobi-Jordan algebra. We have 
$$b_p\leq \Delta_p\leq \Delta-\Delta_a+dim( ~_zH_s^3(\mathfrak J,\mathbb K)).$$ 
\end{corollary}

\begin{remark}
Furthermore, we have 
$$ \Delta_p \geq 1+\frac{m(m+1)}{2}  ,
$$
where $m$ is the dimension of the annihilator of $\mathfrak J$.
\end{remark}
\begin{Example}
      We consider the $4$-dimensional Jacobi-Jordan algebra  $ \mathfrak J_4 $ (Example \ref{ex2}),   and endow it with  pseudo-Euclidean structure $B$, see  \cite{amir2}, where the  invariant scalar product on  $\mathfrak J_4$ is defined by $$B(e_1,e_4) =B(e_2,e_3) =1.$$

      We shall calculate the dimension of $Der( \mathfrak J_4)$. An endomorphism $D$ of $ \mathfrak J_4$ is a derivation if and only if the matrix of $D$ with respect to  the basis  $\{e_1,e_2,e_3,e_4\}$ is written as
      {\begin{small}$$D=\left(
  \begin{array}{cccc}
    \alpha_1 & 0 & 0 & 0 \\
    \alpha_2 & 2\alpha_1 & \alpha_3 & 0 \\
    \alpha_4 & 0 & \alpha_5 & 0 \\
    \alpha_6 & 2\alpha_4 & \alpha_7 & \alpha_1+\alpha_5 \\
  \end{array}
\right).$$
\end{small}} Thus, $dim(Der\mathfrak J_4)=7.$
   We have shown in Example \ref{ex2} that  $~_zZ^2({\mathfrak  J}_4,\mathbb K)$  is spanned by the cocycle corresponding to  the following matrices with respect to  the basis $\{e_1,e_2,e_3,e_4\}$:
    \begin{small}
$$\left(
  \begin{array}{cccc}
    \alpha_{11} & 0 & \alpha_{13} & \alpha_{14} \\
    0 & 0 & -\frac{1}{2}\alpha_{14} & 0 \\
    \alpha_{31} & \alpha_{32} & \alpha_{33} & \alpha_{34} \\
    \alpha_{41} & 0 & 0 & 0 \\
  \end{array}
\right).$$ \end{small}
Consequently,   $dim(~_zZ_a^2({\mathfrak  J}_4,\mathbb K))=2$. By $(ii)$ of  Corollary \ref{corol3}, we have $dim(Der_a(\mathfrak J_4))=2$. Now, we have to compute  $~_zH_s^3({\mathfrak  J}_4,\mathbb K)).$ A basis of  $~_zZ_s^3({\mathfrak  J}_4,\mathbb K))$ is given by 
 $\{c_1,...c_{6}\}$, where the nonzero terms are 
	 \begin{align*}
	    & c_1(e_1,e_3,e_4)=-\frac{1}{4}c_1(e_2,e_3,e_3)=1,\;\\ &
	      c_{2}(e_3,e_3,e_3)=1,\;\\ &
	      c_3(e_1,e_1,e_4)=-c_3(e_1,e_2,e_3)=1,\\
	   & c_{4}(e_1,e_1,e_1)=1,\;\\ &
	      c_5(e_1,e_1,e_3)=1,\;\\ &
	        c_{6}(e_1,e_3,e_3)=1.	        
	 \end{align*}
	 Let $f_1,f_2\in  A^2(\mathfrak J).$ We have
	 $$\delta f_1(e_2,e_3,e_3)=0\neq c_1(e_2,e_3,e_3)=1.$$ 
	 $$\delta f_2(e_3,e_3,e_3)=0\neq c_2(e_3,e_3,e_3)=1.$$ Thus, $c_1$ and $c_2$ are not  coboundaries. Now we prove that $c_3,...,c_6$  are  coboundaries.
	 For $f_3\in  A^2(\mathfrak J)$ defined by $f_3(e_2,e_4)=-f_3(e_4,e_2)=1$ and $f_3(e_i,e_j)=0$ otherwise,  we have $c_3=\delta f_3 $.
	 For $f_4\in  A^2(\mathfrak J)$ defined by  $f_4(e_1,e_2)=-\frac{1}{3}$ and $f_3(e_i,e_j)=0$ otherwise,  we have $c_4=\delta f_4 $.
	 For $f_5\in  A^2(\mathfrak J)$ defined by  $f_5(e_2,e_3)=1$ and $f_3(e_i,e_j)=0$ otherwise,  we have $c_5=\delta f_5$.
	 For $f_6\in  A^2(\mathfrak J)$ defined by $f_6(e_3,e_4)=-\frac{1}{2}$ and $f_6(e_i,e_j)=0$ otherwise,  we have $c_6=\delta f_6 $.
	 Thus, $c_3,...,c_6$  are  coboundaries.
	 Consequently, $dim(~_zH_s^3({\mathfrak  J}_4,\mathbb K))=2$ and $~_zH_s^3({\mathfrak  J}_4,\mathbb K)$ is spanned by the cohomological classes  $c_1$ and $c_2$.
	 By Corollary \ref{dp}, we have:
	 $$dim(F(\mathfrak J_4))\leq 7.$$
	 
\end{Example}
\section{Deformations of Jacobi-Jordan algebras}\label{Sec5}
One-parameter formal deformations were introduced for associative algebras by Gerstenhaber \cite{Gerst} and then studied for Lie algebras by Nijenhuis and Richardson \cite{NijenhuisRichardson}. This approach based on formal power series was extended to various algebraic structures.
In this section we discuss  formal deformations of Jacobi-Jordan algebras and  formal   deformations of  Jacobi-Jordan algebras homomorphisms. We show that the zigzag cohomology defined above is suitable to establish   their connections to cohomology groups.\\

Let $\mathbb{K}[\![t]\!]$ be the power series ring in one variable $t$
and coefficients in $\mathbb{K}$ and ${\mathfrak  J}[\![t]\!]$ be the set of formal power 
series whose coefficients are elements of the vector space ${\mathfrak  J}$, $({\mathfrak  J}[\![t]\!]$ is
obtained by extending the coefficients domain of ${\mathfrak  J}$ from $\mathbb{K}$ to $\mathbb{K}[\![t]\!])$.\\


\begin{definition}
A one-parameter formal deformation of a Jacobi-Jordan algebra $({\mathfrak  J},\cdot )$ is a   Jacobi-Jordan  $\mathbb{K}[\![t]\!]$-algebra $(\mathfrak{J}[\![t]\!],\mu_{t})$,  where $\mu_{t}=\sum_{i\geq 0}^{\infty}t^{i}\mu_{i}$, which is a $\mathbb{K}[\![t]\!]$-bilinear map such that   $ \mu_0(x,y) =x\cdot y$ and 
\begin{align}
    &\mu_t(x,y)=\mu_t(y,x),\label{s1}\\
    &\mu_t(x,\mu_t(y,z))+\mu_t(y,\mu_t(z,x))+\mu_t(z,\mu_t(x,y))=0, \text{ for all } x,y,z\in \mathfrak  J.\label{s2}
\end{align}

The deformation is said to be of order $N$ if  $\mu_{t}=\sum_{i\geq 0}^N\mu_{i}t^{i}$ and infinitesimal if  $N=1$.
\end{definition}
\begin{remark}
\begin{enumerate}
    \item 
Condition (\ref{s1}) is equivalent to the fact that for all $i$,  $\mu_i$ is a symmetric bilinear map.
\item Condition (\ref{s2}) is called deformation equation and  is equivalent to 
$$\sum_{i,j=0}^{\infty}t^{i+j}(\mu_i(x,\mu_j(y,z))+\mu_i(y,\mu_j(z,x))+\mu_i(z,\mu_j(x,y)))=0.$$
By identification with respect to parameter $t$, we get the following  equations for all $p\in\mathbb N$:
$$\sum_{i=0}^p(\mu_i(x,\mu_{p-i}(y,z))+\mu_i(y,\mu_{p-i}(z,x))+\mu_i(z,\mu_{p-i}(x,y)))=0.$$
The first equation ($p=0$) corresponds to Jacobi condition with respect to the initial Jacobi-Jordan multiplication $\mu_0$.
For $p=1$ the equation corresponds to:
\begin{align}\label{s3}
&\mu_0(x,\mu_1(y,z))+\mu_0(y,\mu_1(z,x))+\mu_0(z,\mu_1(x,y))\nonumber \\&+\mu_1(x,\mu_0(y,z))+\mu_1(y,\mu_0(z,x))+\mu_1(z,\mu_0(x,y))=0.
\end{align}

If we consider $\mathfrak  J$ as a $\mathfrak  J$-module with respect to the adjoint action, we get that the equation (\ref{s3})  is equivalent to the fact that $ d^2\mu_1(x,y,z)=0,$ for all $x,y,z\in\mathfrak  J$.
\end{enumerate}
\end{remark}
Hence, we have  the following proposition. 
\begin{proposition}\label{-s8}
Let $(\mathfrak  J,\mu_0)$ be a Jacobi-Jordan algebra and $\mu_t=\sum_i t^i\mu_i$ be a deformation of $\mathfrak  J$.  Then, $$\mu_1\in ~_zZ_s^2(\mathfrak  J,\mathfrak  J).$$
\end{proposition}
\begin{definition}
Let $(\mathfrak  J,\mu_0)$ be a Jacobi-Jordan algebra and let $\mu_t=\sum t^i\mu_i$ and $\mu'_t=\sum t^i\mu'_i$ be two deformations of $\mathfrak  J$ such that $ \mu_0= \mu'_0$.
We say that $\mu$ and $\mu'$ are equivalent if there exists a formal isomorphism of the form $$\Phi_t=\Phi_0+t\Phi_1+t^2\Phi_2+...., \text{ where } \Phi_0=id,$$ such that 
\begin{align}   \label{s4}
\Phi_t\circ \mu'_t=\mu_t\circ( \Phi_t\otimes{\Phi_t}).
\end{align}
Expanding (\ref{s4}) and identifying coefficient of $t^s$ lead to:
\begin{align}\label{s5}
    \sum_{i+j=s}\Phi_i \mu'_j-\sum_{i+j+k=s}\mu_i\circ( \Phi_j\otimes{\Phi_k})=0.
    \end{align}
\end{definition}
Notice that equation for  $s=0$ since   $ \mu_0= \mu'_0$ because $\Phi_0=id$.
\begin{proposition}\label{s17}
Let $(\mathfrak  J,\mu_0)$ be a Jacobi-Jordan algebra. Let   $\mu_t=\sum t^i\mu_i$ and $\mu'_t=\sum t^i\mu'_i$ be two equivalent deformations of $\mathfrak  J$. Then $\mu'_1$ and $\mu_1$ are cohomologous.
\end{proposition}
\begin{proof}
For $s=1$ in (\ref{s5}), one finds that 
$$ \mu'_1+\Phi_1\mu'_0-\mu_1-\mu_0(\Phi_1\otimes id)-\mu_0(id \otimes \Phi_1)=0,$$
which is equivalent to 
\begin{align} \label{s6}
\mu'_1=\mu_1+\mu_0(\Phi_1\otimes id)+\mu_0(id \otimes \Phi_1)-\Phi_1\mu_0=0.
\end{align}
Recall that $\delta^1\Phi_1=\mu_0(\Phi_1\otimes id)+\mu_0(id \otimes \Phi_1)-\Phi_1\mu_0.$ 
Then the equality (\ref{s6}) is equivalent to:
$$\mu'_1=\mu_1+\delta^1\Phi_1.$$
\end{proof}
Therefore the first term of a deformation depends only on its cohomology class with respect to zigzag cohomology of a Jacobi-Jordan algebra.
\begin{proposition}
 Let $(\mathfrak  J,\mu_0)$ be a Jacobi-Jordan algebra. There is, over $K[[t]]/t^2,$ a one-to-one correspondence between  elements $\mu_1$ of $~_zH_s^2({\mathfrak  J},{\mathfrak  J})$ and  inﬁnitesimal deformations of $\mathfrak  J$ deﬁned by
 $$\mu_t(x,y)=\mu_0(x,y)+t\mu_1(x,y), \text{ for all }x,y \in\mathfrak  J. $$
 \end{proposition}
\begin{proof}
The proof follows from Propositions \ref{-s8} and  \ref{s17}.
\end{proof}

\section{Deformations of Jacobi-Jordan algebra homomorphisms}\label{Sec6}

Let $\mathfrak  J$ be a Jacobi-Jordan algebra, $A$ be a Jacobi-Jordan admissible  algebra and $\varphi:{\mathfrak  J}\longrightarrow A$ be a Jacobi-Jordan homomorphism, where $A$ is endowed by the Jacobi-Jordan product $\{.,.\}$ induced by the anti-commutator. We consider $A$ as a $\mathfrak  J$-module with the action defined by $x.a:=-L_{\varphi(x)}a=-\{\varphi(x),a\}$, where $L$ is  the left multiplication of $(A,\{.~,.\})$.

A formal deformation of $\varphi$ is a formal power series 
\begin{eqnarray}\label{eq--0}
\Phi_t=\sum_{r\geq 0}t^r\Phi_r, \mbox{ where } \Phi_0=\varphi ,
\end{eqnarray}  and $\Phi_r \; (r\geq 1)$  are  linear maps from $\mathfrak  J$ into $A[[t]]$, a formal space with coefficients in $A$,  such that for all $x$ and $y$ in $\mathfrak  J$, we have: 
\begin{eqnarray}\label{eq1} \Phi_t(xy)=\{\Phi_t(x),\Phi_t(y)\}.
\end{eqnarray}The product of the right hand side  is the anti-commutator extended to formal power series by $\mathbb{K}[[t]]$-bilinearity. The identity $(\ref{eq1})$ is equivalent to an infinite system of equation of the form:
\begin{eqnarray}\label{eq2}\Phi_n(xy)=\sum_{k=0}^n\{\Phi_k(x),\Phi_{n-k}(y)\},\mbox{ for all }x,y\in\mathfrak  J\mbox{ (Deformation equation of order $n$,  $n\in \mathbb{N}$)}.
\end{eqnarray}
Notice that $n=0$ corresponds to $\Phi_0=\varphi$ being a Jacobi-Jordan homomorphism.
\begin{remark}
 The deformation equation (\ref{eq2}) of order  $n\geq1$ can be written as follow:
$$d^1\Phi_n(xy)=-\sum_{k=1}^{n-1}\{\Phi_k(x),\Phi_{n-k}(y)\},$$
where $d^1\Phi_n(xy)=\{\varphi(x),\Phi_{n}(y)\}+\{\Phi_{n}(x),\varphi(y)\}-\Phi_n(x y)$. 
In particular $\Phi_1$ is a $1$-cocycle.
\end{remark}
\begin{definition} Two formal deformations  $\Phi_t$ and $\Phi'_t$ of $\varphi$ are
equivalent  if there exists an automorphism  
$f_t: A [[t]]\longrightarrow A [[t]]$ of the form
$$f_t=exp(tL_{\varphi(a_1)}+t^2L_{\varphi(a_2)}+....)=id+tL_{\varphi(a_1)}+t^2(L^2_{\varphi(a_1)}/2+L_{\varphi(a_2)})+....$$
where $a_i\in A$, such that $\Phi'_t=f_t\Phi_t$ and $L $ is the left multiplication of $(A,\{.~,.\})$.
\end{definition}

Recall that a deformation of a Jacobi-Jordan algebra homomorphism is said infinitesimal if it is of  of the form $\Phi_t(x,y)=\Phi_0(x,y)+t\Phi_1(x,y), \text{ for all }x,y \in\mathfrak  J $.

\begin{theorem}
Infinitesimal deformations of a Jacobi-Jordan  algebra
homomorphism from $\mathfrak  J$  into $A$ are classified by the first cohomology group $~_zH^1({\mathfrak  J},A).$
\end{theorem}
\begin{proof}
The first order terms $\Phi_1$ in $(\ref{eq--0})$ are 1-cocycles.
Consider two infinitesimal  formal deformations  $\Phi$ and $\Phi'$ of $\varphi$. 
They are equivalent if and only if the corresponding cocycles are
cohomologous.
Conversely, given a Jacobi-Jordan  algebra homomorphism $\varphi : \mathfrak  J\longrightarrow A$, an arbitrary 1-cocycle $\Phi_1\in ~_zZ^1(\mathfrak  J; A) $ defines an infinitesimal deformation of $\varphi$.
\end{proof}

In a forthcoming paper, we will extend the zigzag cohomology to  provide a cohomology that controls  simultaneous deformations of Jacobi-Jordan algebra and Jacobi-Jordan algebra homomorphism.

\section{Example : The space $~_zH_s^2({\mathfrak  J}_4,{\mathfrak  J}_4)$ and Deformations}\label{Sec7}
We consider  the $4$-dimensional Jacobi-Jordan algebra  $({\mathfrak  J}_4,\cdot)$ presented in Example \ref{ex2}. We aim in this section to compute its second zigzag cohomology group $~_zH_s^2({\mathfrak  J}_4,{\mathfrak  J}_4)$ and consider some deformations. In particular, we provide a  nonisomorphic deformation of  ${\mathfrak  J}_4$.
	
	We first describe the second zigzag cohomology group.
	
		\begin{proposition} 
	Let $({\mathfrak  J}_4,\cdot)$ be the $4$-dimensional Jacobi-Jordan algebra defined in Example \ref{ex2}. Then,
	 $\dim ~_zZ_s^2({\mathfrak  J}_4,{\mathfrak  J}_4)=13$ and  $~_zZ_s^2({\mathfrak  J}_4,{\mathfrak  J}_4)$ is generated, with respect to the basis $\{e_1,...,e_{4}\}$, by   $\{c_1,...c_{13}\}$, where the nonzero terms are defined as 
	 \begin{eqnarray*}
	   & c_{1}(e_3,e_3)=e_2,
	   & c_2(e_1,e_3)=e_1,~ c_2(e_2,e_3)=-2e_2,\\
	   & c_3(e_1,e_3)=e_3, ~c_3(e_2,e_3)=-2e_4,
	 & c_4(e_1,e_4)=e_2,~ c_4(e_2,e_3)=-2e_2,\\
	 & c_{5}(e_1,e_4)=e_4,~ c_{5}(e_2,e_3)=-2e_4, &  c_{6}(e_3,e_3)=2e_3, c_{6}(e_3,e_4)=-e_4, \\ & c_7(e_1,e_1)=e_1,~c_7(e_1,e_2)=-e_2,  & 
	   c_8(e_1,e_1)=e_2,
	    \\ & c_9(e_1,e_1)=e_3 , c_9(e_1,e_2)=-e_4, &
	    c_{10}(e_1,e_1)=e_4,\\ &  c_{11}(e_1,e_3)=e_2, & c_{12}(e_1,e_3)=e_4,\\ &
	    c_{13}(e_3,e_3)=e_4. & \ 
	 \end{eqnarray*}

	\end{proposition}

	\begin{proof}
A symmetric 2-cochain $c$ belongs to $~_zZ_s^2({\mathfrak  J}_4,{\mathfrak  J}_4)$ if and only if $c$ satisfies $d^2c=0$ which may be expressed with respect to the basis by the following conditions
\begin{align}\label{s8}
    c(e_ie_j,e_k)+c(e_je_k,e_i)+c(e_ke_i,e_j)+e_ic(e_j,e_k)+e_jc(e_k,e_i)+e_kc(e_i,e_j)=0,
\end{align} where  $i,j,k\in\{1,2,3,4\}.$
In the sequel, we  suppose that 
$c(e_i,e_j)=\sum_{l=1}^4 \alpha_{i,j}^l e_l,\text{ for }i,j\in\{1,2,3,4\}.$
From the  identity (\ref{s8}), we get the following equations on the coefficient $\alpha_{i,j}^l :$ For $(i,j,k)=(1,1,1)$ and $(i,j,k)=(1,1,2)$, we obtain respectively:
	\begin{align}
	   & \alpha_{1,2}^1= \alpha_{1,2}^3=0,\,\,\,\, \alpha_{1,2}^2= -\alpha_{1,1}^1\text{ and } \alpha_{1,2}^4=- \alpha_{1,1}^3.\label{alpha1}\\
	   & \alpha_{2,2}^1=\alpha_{2,2}^3=0,\alpha_{2,2}^2= -2\alpha_{1,2}^1, \alpha_{2,2}^4=-2\alpha_{1,2}^3,
	      \end{align} 
which implies that $\alpha_{2,2}^l=0.$
 For  $(i,j,k)=(1,1,3)$ and  $(i,j,k)=(1,1,4)$, we obtain respectively: 
\begin{align}
&\alpha_{2,3}^1+2\alpha_{1,4}^1=0,~
\alpha_{2,3}^2+2\alpha_{1,4}^2+2\alpha_{1,3}^1=0,~
\alpha_{2,3}^3+2\alpha_{1,4}^3=0,\text{ and }
\alpha_{2,3}^4+2\alpha_{1,4}^4+2\alpha_{1,3}^3+\alpha_{1,1}^1=0.\label{alpha2}\\
 &\alpha_{2,4}^1=\alpha_{2,4}^3=0,\,\,\,\alpha_{2,4}^2+2\alpha_{1,4}^1=0\text{ and } \alpha_{2,4}^4+2\alpha_{1,4}^3=0.\label{alpha3}
\end{align}
\begin{align*}
&\text{ For } (i,j,k)=(1,2,3), \text{ we get }\alpha_{2,4}^1=\alpha_{2,4}^3=0,~\alpha_{2,4}^2+\alpha_{2,3}^1=0 \text{ and } \alpha_{2,4}^4+\alpha_{2,3}^3+\alpha_{1,2}^1=0. \\
&\text{ By } (\ref{alpha1})\text{  and } (\ref{alpha2}),\text{  we have } \alpha_{2,3}^3=-2\alpha_{1,4}^3\text{ and } \alpha_{1,2}^1=0.\text{ Thus, } \alpha_{2,4}^4=-2\alpha_{1,4}^3.
\end{align*}
For $(i,j,k)=(1,3,3),(1,3,4),(2,3,3)$, we obtain the  following equations:
\begin{align}
   &\alpha_{3,4}^1=\alpha_{3,4}^3=0,\,\,\, 2\alpha_{3,4}^2+\alpha_{3,3}^1=0 \text{ and }2\alpha_{3,4}^4+2\alpha_{1,3}^1+\alpha_{3,3}^3=0. \label{alpha4}\\
   &\alpha_{4,4}^1=\alpha_{4,4}^3=0,  ~\alpha_{4,4}^2+\alpha_{3,4}^1=0 \text{ and }\alpha_{4,4}^4+\alpha_{3,4}^3+\alpha_{1,4}^1=0.\\
   &\alpha_{2,3}^1=0.
\end{align}
The equation  (\ref{alpha2}) entails that $\alpha_{1,4}^1=0$  and the equations (\ref{alpha3}) and (\ref{alpha4}) entails that $\alpha_{2,4}^2=0$ and $\alpha_{4,4}^4=0$.
$\text{ For } i=j=k=3, \text{ we get }e_3c(e_3,e_3)=0. \text{ which implies } \alpha_{3,3}^1=0.
$
The other cases come from the conditions above and from the fact that $e_2$ and $e_4$ belong to the annihilator of $\mathfrak  J_4$ without other restriction on $\alpha_{i,j}^l$.

Therefore for all $x=x_1e_1+x_2e_2+x_3e_3+x_4e_4$ and $y=y_1e_1+y_2e_2+y_3e_3+y_4e_4$ in ${\mathfrak  J}_4$, we have:
\begin{align*}
c(x,y)=&\alpha_{1,1}^1\Bigl(x_1y_1e_1-(x_1y_2+x_2y_1)e_2-(x_2y_3+x_3y_2)e_4\Bigl)
+\alpha_{1,1}^2\Bigl(x_1y_1e_2\Bigl)\\
&+\alpha_{1,1}^3\Bigl(x_1y_1e_3-(x_1y_2+x_2y_1)e_4\Bigl))+\alpha_{1,1}^4\Bigl((x_1y_1)e_4\Bigl)\\
&+\alpha_{1,3}^1\Bigl((x_1y_3+x_3y_1)e_1-2(x_2y_3+x_3y_2)e_2-(x_3y_4+x_4y_3)e_4\Bigl)\\
&+\alpha_{1,3}^2\Bigl((x_1y_3+x_3y_1)e_2\Bigl)+\alpha_{1,3}^3\Bigl((x_1y_3+x_3y_1)e_3-2(x_2y_3+x_3y_2)e_4\Bigl)\\
&+\alpha_{1,3}^4\Bigl((x_1y_3+x_3y_1)e_4\Bigl)+\alpha_{1,4}^2\Bigl((x_1y_4+x_4y_1-2x_2y_3-2x_3y_2)e_2\Bigl)\\
&+\alpha_{1,4}^4\Bigl((x_1y_4+x_4y_1-2x_2y_3-2x_3y_2)e_4\Bigl)+\alpha_{3,3}^2\Bigl(x_3y_3e_2\Bigl)\\
&+
\alpha_{3,3}^3\Bigl((2x_3y_3 e_3-(x_3y_4+x_4y_3)e_4\Bigl)+\alpha_{3,3}^4\Bigl(x_3y_3e_4\Bigl)
\end{align*}
which  entails that $\{c_1,c_2,...c_{13}\}$ is a generating subset of $~_zZ_s^2({\mathfrak  J}_4,{\mathfrak  J}_4)$. Moreover, it is clear that   $\{c_1,c_2,...c_{13}\}$ is a free family of cocycles.
	\end{proof}

Let  us now compute the $2$-coboundaries of the family $\{c_1,c_2,...c_{13}\}$ in order to determine the dimension and a basis of $~_zH_s^2({\mathfrak  J}_4,{\mathfrak  J}_4)$. It is
 stated in the following   main theorem of this section.

\begin{theorem}
$\dim\Bigl( ~_zH_s^2({\mathfrak  J}_4,{\mathfrak  J}_4)\Bigl)=6$ and $~_zH_s^2({\mathfrak  J}_4,{\mathfrak  J}_4)=span\{c_1,...,c_{6}\}.$
\end{theorem}
\begin{proof} First, we prove that $c_7,...,c_{13}$ are  coboundaries.
Recall that a cocycle $c\in ~_zZ_s^2({\mathfrak  J}_4,{\mathfrak  J}_4) $ is a coboundary in case there exists a
$1$-cochain $f$ such that:
$$c(e_i, e_j ) =e_if(e_j)+e_jf(e_i)- f(e_ie_j),\text{ for all }i\in\{1,2,3,4\}.$$
Using this equation, we fill in the table of cochains and corresponding  coboundaries:
\vskip 0.5cm

 \begin{tabular}{|l|*{2}{c|}}
    \hline
     cochain  &coboundary   \\
    \hline
     $f_7(e_2)=-e_1,\,\,f_7(e_1)=f_7(e_3)=f_7(e_4)=0$&$ \delta^1f_7=c_7 $ \\
    \hline
     $f_8(e_2)=-e_2,\,\,f_8(e_1)=f_8(e_3)=f_8(e_4)=0 $& $\delta^1f_8=c_8$ \\ 
    \hline
      $f_9(e_2)=-e_3,\,\,f_9(e_1)=f_9(e_3)=f_9(e_4)=0 $& $\delta^1 f_9=c_9$ \\ 
    \hline
      $f_{10}(e_2)=-e_{1},\,\,f_{10}(e_1)=f_{10}(e_3)=f_{10}(e_4)=0$ &$ \delta^1f_{10}=c_{10} $\\ 
    \hline
     $f_{11}(e_4)=-e_2 \,\,f_{11}(e_1)=f_{11}(e_2)=f_{11}(e_3)=0$&$\delta^1f_{11}= c_{11}$ \\
    \hline
   $ f_{12}(e_4)=-e_4,\,\,f_{12}(e_1)=f_{12}(e_2)=f_{12}(e_3)=0  $ &$ \delta^1f_{12}=c_{12} $ \\
    \hline
     $ f_{13}(e_3)=1/2~e_1,\,\,f_{13}(e_4)=1/2~e_2,\,\,f_{13}(e_1)=f_{13}(e_2)    =0  $ & $\delta^1f_{13}=c_{13} $ \\
    \hline
\end{tabular}
\vskip 0.5cm
Now, we explain how $c_1,...,c_{6}$ are not trivial, case by case, cocycle by cocycle.

For $i=1,...,6$, an eventual  corresponding cochain of $c_{i}$ must satisfy respectively the following equation $Eq_i$:
\begin{align*}
&Eq_1:2e_3f(e_3)=e2, ~~Eq_2:e_3f(e_2)=-2e_2, Eq_3:e_3f(e_2)=-2e_4 \text{ and }
2e_1f(e_1)-f(e_2)=0,\\
&Eq_4:e_3f(e_2)=-2e_2,~~Eq_5:e_1f(e_2)=-2e_4\text{ and }2e_1f(e_1)-f(e_2)=0,~~Eq_6:e_3f(e_3)=e_3.
\end{align*}
By easy computation, we can verify that each equation $Eq_i$ has no solution.
 Thus, $c_{1},...,c_6$ are not trivial.
 Now, we prove that  the cohomology classes of  $c_1,...,c_6$ are linearly independent.  Let $\alpha_1,...,\alpha_6$ be  scalars and $f$ be an endomorphism of $\mathfrak J_4$ such that 
 \begin{align}\label{comb}
(\alpha_1c_1+...+\alpha_6c_6)(e_i,e_j)=\delta^1f(e_i,e_j),
 \end{align}
  for all $i,j\in\{1,...,4\}$.\\ Let $M=\Bigl(a_{ij}\Bigl)_{1\leq i,j\leq 4}$ be the matrix of $f$ with respect to  the basis $\{e_1,...,e_4\}$.

 For $i=j=3$,  Equation (\ref{comb}) holds if and only if
   $\alpha_1=\alpha_6=0$.
   
  For $i=2,~j=3$,  Equation (\ref{comb})  implies that  $\alpha_2+\alpha_4=0$ and $\alpha_3+\alpha_5=-\frac{1}{2}a_{12}$.
   For $i=1,~j=4$,  Equation (\ref{comb})   implies that  $\alpha_4=a_{14}$ and $\alpha_5=a_{34}$.
    For $i=1,~j=3$,  Equation (\ref{comb})   implies that  $\alpha_2=a_{14}$ and $\alpha_3=a_{34}$. Thus, $\alpha_2=\alpha_4=0$. Moreover,
       For $i=1,~j=2$,  Equation (\ref{comb})   implies that  $a_{12}=0$. Thus,  $\alpha_3=\alpha_5=0$. 
 
\end{proof}
\begin{Example}[Deformation of $\mathfrak  J_4$] We aim to a make a use of the cocycles to construct a deformation of $\mathfrak  J_4$.  
       Consider  the first  cocycle $c_1$ defined  above where the non-zero term is
       $ c_{1}(e_3,e_3)=e_2$.  Since 
       $\sum_{cyc\{x,y,z\}}c_1(c_1(x,y),z)=0,$ then we can consider
        the formal deformation of the Jacobi-Jordan algebra $\mathfrak  J_4$ defined  by 
       $x\star_ty:=xy+tc_1(x,y).$
       The new product $\star_t$ becomes defined by:
       $e_1\star_te_1=e_2,~~e_1\star_te_3=e_3\star_te_1=e_4,~~e_3\star_te_3=te_2, \text{ otherwise } e_i\star_te_j=0.$ It defines a new Jacobi-Jordan algebra. 
       If we consider the new basis
       $u_1:= e_3-\sqrt{t}e_1,~~u_2:=(t^2+t)e_2-2\sqrt{t}e_4,~~u_3:= e_3+\sqrt{t}e_1,~~u_4:=(t^2+t)e_2+2\sqrt{t}e_4.$
       The non zero terms of the product $\star_t$ are:
       $u_1u_1=u_2,~~u_3u_3=u_4\text{ otherwise } u_i\star_tu_j=0.$
       Thus  the deformed algebra  $(\mathfrak  J_4,\star_t)$ is isomorphic to the algebra $A_{1,2}\oplus A_{1,2}$ of Proposition 3.4 in \cite{burde}.
       
\end{Example}

\section*{ACKNOWLEDGEMENTS}
The work is funded by grant number 15-MAT5345-10 from the
National Science, Technology and Innovation Plan (MAARIFAH), the King Abdul-Aziz City for Science and Technology (KACST), Kingdom of Saudi Arabia. We thank the Science and Technology Unit at
Umm Al-Qura University for their continued logistics support. 

\end{document}